\documentclass{amsart}

\usepackage{amsfonts, amssymb, amsmath, eucal, verbatim, amsthm, amscd, enumerate}
\usepackage{setspace}

\newtheorem{theorem}{Theorem}[section]
\newtheorem{corollary}[theorem]{Corollary}
\newtheorem{lemma}[theorem]{Lemma}

\theoremstyle{definition}

\theoremstyle{remark}
\newtheorem*{remark}{Remark}

\newtheorem*{acknowledgement}{Acknowledgement}

\numberwithin{equation}{section}
\def\R{{\mathbb R}}

\begin{document}

\author{Jonathan Bennett}
\address{Jonathan Bennett, School of Mathematics, University of Birmingham, Edgbaston, Birmingham, B15 2TT, UK}
\email{j.bennett@bham.ac.uk}

\author{Neal Bez} 
\address{Neal Bez, Department of Mathematics, Graduate School of Science and Engineering,
Saitama University, Saitama 338-8570, Japan}
\email{nealbez@mail.saitama-u.ac.jp}

\author{Chris Jeavons} 
\address{Chris Jeavons, School of Mathematics, University of Birmingham, Edgbaston, Birmingham, B15 2TT, UK}
\email{jeavonsc@maths.bham.ac.uk}

\author{Nikolaos Pattakos}
\address{Nikolaos Pattakos, School of Mathematics, University of Birmingham, Edgbaston, Birmingham, B15 2TT, UK}
\email{nikolaos.pattakos@gmail.com}

\date{\today}


\title{On sharp bilinear Strichartz estimates of Ozawa--Tsutsumi type}
\maketitle \thispagestyle{empty}

\begin{abstract}
We provide a comprehensive analysis of sharp bilinear estimates of Ozawa--Tsutsumi type for solutions $u$
of the free Schr\"odinger equation, which give sharp control on $|u|^2$ in classical Sobolev spaces. In particular, we generalise their estimates in such a way that provides
a unification with some sharp bilinear estimates proved by Carneiro and Planchon--Vega, via entirely different methods, by
seeing them all as special cases of a one-parameter family of sharp estimates. The extremal functions are solutions of the Maxwell--Boltzmann functional equation and hence Gaussian. 
For $u^2$ we argue that the natural analogous results involve certain dispersive Sobolev norms; in particular, despite the validity of the classical Ozawa--Tsutsumi estimates for both $|u|^2$ and $u^2$ in the classical Sobolev spaces, we show that Gaussians are not extremisers in the latter case for spatial dimensions strictly greater than two.
\end{abstract}

\section{Introduction}
For $d \geq 2$, consider the free Schr\"odinger equation
\begin{align} \label{e:Schro}
i\partial_t u +  \Delta u = 0, \qquad u(0) = u_0
\end{align}
on $\mathbb{R}^{1+d}$ with initial data $u_0 \in L^2(\mathbb{R}^d)$. In \cite{OT}, Ozawa and Tsutsumi showed that any two solutions $u$ and $v$ of \eqref{e:Schro} with initial data $u_0$ and $v_0$, respectively, satisfy the global space-time bilinear estimate
\begin{equation} \label{e:OT}
\| (-\Delta)^{\frac{2-d}{4}} (u \overline{v}) \|_{L^2}^2 \leq \mathbf{OT}(d)  \|u_0\|_{L^2}^2 \|v_0\|_{L^2}^2,
\end{equation}
where
\[
\mathbf{OT}(d) = \frac{2^{-d} \pi^{\frac{2-d}{2}}}{\Gamma(\tfrac{d}{2})}\,.
\]
They also showed that the constant $\mathbf{OT}(d)$ is optimal by observing that if $u_0(x) = v_0(x) = \exp(-|x|^2)$ then \eqref{e:OT} is an equality; i.e. $(u_0,v_0)$ is an extremising pair of initial data.

The case of one spatial dimension is rather special and in this case \eqref{e:OT} is true as an identity
\begin{equation} \label{e:OTidentity}
\| (-\Delta)^{\frac{1}{4}} (u \overline{v}) \|^2_{L^2} = \mathbf{OT}(1)  \|u_0\|^2_{L^2} \|v_0\|^2_{L^2}
\end{equation}
for any $(u_0,v_0) \in L^2(\mathbb{R}) \times L^2(\mathbb{R})$. This identity was established in \cite{OT}, gives control on the so-called null gauge form $\partial(u\overline{v})$ for the Schr\"odinger equation in one spatial dimension, and was used as a tool in the proof of local well-posedness of some nonlinear Schr\"odinger equations with nonlinearities involving $\partial(|u|^2)u$.

In the case where $u_0$ is equal to $v_0$, one may view the estimate \eqref{e:OT} as a replacement, in the case where the initial data is in $L^2(\mathbb{R}^d)$, for the Sobolev--Strichartz estimate
\[
\| |u|^2 \|_{L^2}^2 = \| u\|_{L^4}^4 \leq C_d \| u_0 \|_{\dot{H}^{\frac{d-2}{4}}}^4
\]
which requires rather more regularity on the initial data as the dimension gets large. Here, $\dot{H}^s$ denotes the homogeneous Sobolev space with norm
\begin{equation*}
\|f\|_{\dot{H}^s} = \| (-\Delta)^{\frac{s}{2}} f \|_{L^2}.
\end{equation*}
This ``trade-off" of derivatives on the initial data on the right-hand side for derivatives on the square of the solution on the left-hand side was studied by Klainerman and Machedon \cite{KM} for solutions of the homogeneous wave equation; see \cite{FK} for a systematic study of such phenomena in the context of the wave equation, allowing also so-called hyperbolic derivatives on the left-hand side corresponding to the space-time Fourier multiplier $| |\tau| - |\xi||$.

Recently there has been considerable interest in obtaining optimal constants and the existence/shape of extremising initial data associated with space-time estimates for solutions of \eqref{e:Schro} and dispersive equations more widely. For example, it is known that if $u$ solves \eqref{e:Schro} in one spatial dimension then
\begin{equation} \label{e:FHZd=1}
\| u\|_{L^6(\mathbb{R}^{1+1})} \leq \frac{1}{12^{\frac{1}{12}}} \|u_0\|_{L^2}
\end{equation}
and in two spatial dimensions
\begin{equation} \label{e:FHZd=2}
\| u\|_{L^4(\mathbb{R}^{1+2})} \leq \frac{1}{2^{\frac{1}{2}}} \|u_0\|_{L^2}.
\end{equation}
In each case \eqref{e:FHZd=1} and \eqref{e:FHZd=2}, the constant is optimal since there is equality when $u_0(x) = \exp(-|x|^2)$. These sharp estimates were proved by Foschi \cite{Foschi} and also Hundertmark and Zharnitsky \cite{HZ}; we also note that \eqref{e:FHZd=2} follows from \eqref{e:OT} in the case $d=2$ by choosing $u_0 = v_0$, which means we have a number of proofs of this sharp estimate (see also the proofs in \cite{BBCH} and \cite{BBI}, where the emphasis is on underlying heat-flow monotonicity phenomena).

If $u_0$ is an extremiser for either \eqref{e:FHZd=1} or \eqref{e:FHZd=2} then, up to the action of certain transformations, $u_0$ must be an isotropic centred Gaussian. This complete characterisation of the set of extremising initial data (which can be found in \cite{Foschi} or \cite{HZ}; see also \cite{JiangShao} for an alternative proof for $d=2$) was used in \cite{DMR} to establish some impressive results on sharp Strichartz norms for solutions of the mass-critical nonlinear Schr\"odinger equation in spatial dimensions one and two.

Based on the approach in \cite{HZ}, Carneiro proved in \cite{Carneiro} that any two solutions $u$ and $v$ of \eqref{e:Schro} satisfy
\begin{equation} \label{e:Carneiro}
\| u v \|_{L^2}^2 \leq \mathbf{C}(d) \int_{\mathbb{R}^{2d}} |\widehat{u}_0(\zeta)|^2|\widehat{v}_0(\eta)|^2 |\zeta - \eta|^{d-2} \, \mathrm{d}\zeta \mathrm{d}\eta,
\end{equation}
where $d \geq 2$ and
\begin{equation} \label{e:Carneiroconstant}
\mathbf{C}(d) = \frac{2^{2-4d} \pi^{\frac{2-5d}{2}}}{\Gamma(\frac{d}{2})}.
\end{equation}
It was shown in \cite{Carneiro} that the constant in \eqref{e:Carneiro} is optimal and $(u_0,v_0)$ is an extremising pair if and only if $u_0(x) = v_0(x) = \exp(-|x|^2)$, up to certain transformations. Since we are dealing with explicit constants, we should clarify that we take the following Fourier transform
\[
\widehat{f}(\xi) = \int_{\mathbb{R}^d} f(x) \exp(-ix \cdot \xi) \,\mathrm{d}x.
\]
A very closely related bilinear estimate
\begin{equation} \label{e:PV}
\|(-\Delta)^{\frac{3-d}{4}}(u\overline{v}) \|_{L^2}^2 \leq \mathbf{PV}(d) \int_{\R^{2d}}|\widehat{u}_0(\zeta)|^2|\widehat{v}_0(\eta)|^2 |\zeta - \eta| \, \mathrm{d}\zeta \mathrm{d}\eta
\end{equation}
for solutions $u$ and $v$ of \eqref{e:Schro} and $d \geq 2$ is a particular case of some far-reaching identities proved by Planchon and Vega in \cite{PV} using an innovative and radically different approach to those in \cite{Carneiro}, \cite{Foschi}, \cite{HZ} and \cite{OT}. Here, the constant $\mathbf{PV}(d)$ is given by
\[
\mathbf{PV}(d) = \frac{2^{-3d}\pi^{\frac{1-5d}{2}}}{\Gamma(\frac{d+1}{2})}
\]
and can be shown to be optimal. The emphasis in \cite{PV} is not on establishing optimal constants and identifying extremisers; in fact, the explicit constant in \eqref{e:PV} and its optimality, and a characterisation of the set of extremising initial data were not discussed.

Our first main result is a unification of \eqref{e:OT}, \eqref{e:Carneiro} and \eqref{e:PV} by seeing these sharp estimates as special cases of a one-parameter family of sharp estimates. Varying this parameter represents to a trade-off of lowering the exponent on the kernel $|\zeta- \eta|$ on the right-hand side, which may be viewed as lowering the ``derivatives" on the right-hand side, with a lowering of the order of derivatives on $|u|^2$ on the left-hand side (very much in the spirit of \cite{KM}). Interestingly, we show that the extremising initial data must satisfy the so-called Maxwell--Boltzmann functional equation. This functional equation arises in the proof of Boltzmann's $H$-theorem in connection with the derivation of hydrodynamic equations from Boltzmann's equation and is known to admit only Gaussian solutions under the assumption that the input functions are integrable. From this we will deduce that the extremisers for our sharp estimates are Gaussians.

Before the statement, we introduce a little notation. First, we write
\[
\Upsilon_\lambda := \{ (f,g) : \text{$f,g : \mathbb{R}^d \to \mathbb{C}$ measurable and $I_\lambda(f,g) < \infty$}\}\,,
\]
where
\[
I_\lambda(f,g) := \int_{\R^{2d}} |\widehat{f}(\zeta)|^2 |\widehat{g}(\eta)|^2  |\zeta - \eta|^{4\lambda+d-2} \, \mathrm{d}\zeta \mathrm{d}\eta\,.
\]
Also, we let $\mathfrak{G}$ denote the class of Gaussian functions on $\mathbb{R}^d$ given by
\[
\mathfrak{G} := \{  \exp(a|\eta|^2 + b \cdot \eta + c) : \mbox{$a, c \in \mathbb{C}$, $b \in \mathbb{C}^d$ and $\mbox{Re}(a) < 0$}\}
\]
and $\mathfrak{G}_r$ will denote the subclass with radial modulus; that is,
\[
\mathfrak{G}_r := \{  \exp(a|\eta|^2 + i b \cdot \eta + c) : \mbox{$a, c \in \mathbb{C}$, $b \in \mathbb{R}^d$ and $\mbox{Re}(a) < 0$}\}.
\]

\begin{theorem} \label{t:OTCunify}
Let $d \geq 2$ and $\sigma>\frac{1-d}{4}$. Then
\begin{equation} \label{e:OT''}
\|(-\Delta)^{\sigma}(u\overline{v}) \|_{L^2}^2 \leq \mathbf{OT}(d,\sigma) \int_{\R^{2d}}|\widehat{u}_0(\zeta)|^2|\widehat{v}_0(\eta)|^2 |\zeta - \eta|^{4\sigma+d-2} \, \mathrm{d}\zeta \mathrm{d}\eta
\end{equation}
for solutions $u$ and $v$ of \eqref{e:Schro} with initial data $(u_0,v_0) \in \Upsilon_{\sigma}$. Here,
\[
\mathbf{OT}(d,\sigma) = 2^{-3d} \pi^{\frac{1-5d}{2}} \frac{\Gamma(2\sigma+\frac{d-1}{2})}{\Gamma(2\sigma + d-1)}
\]
is the optimal constant which is attained if and only if $\widehat{u}_0 \in \mathfrak{G}$ and $v_0$ is a scalar multiple of $u_0$.
\end{theorem}

We follow this statement with several remarks. Firstly, we have
\begin{equation*}
\mathbf{OT}(d,\sigma) = \left\{\begin{array}{lllll} (2\pi)^{-2d}\mathbf{OT}(d) & \text{if $\sigma = \frac{2-d}{4}$} \\
\mathbf{C}(d) & \text{if $\sigma = 0$} \\
\mathbf{PV}(d) & \text{if $\sigma = \frac{3-d}{4}$,}
\end{array} \right.
\end{equation*}
where the expression for $\sigma = 0$ can be verified using the duplication formula,
\[
\Gamma(z)\Gamma(z+\tfrac{1}{2}) = 2^{1-2z}\sqrt{\pi}\Gamma(2z)
\]
for the Gamma function. Hence, when $\sigma = \frac{2-d}{4}$, estimate \eqref{e:OT''} obviously coincides with \eqref{e:OT} after an application of Plancherel's theorem on the right-hand side. When $\sigma = 0$, \eqref{e:OT''} coincides with \eqref{e:Carneiro} since once the operator $(-\Delta)^\sigma$ disappears, the complex conjugate on $v$ has no effect (we will soon see that for $\sigma \neq 0$, the complex conjugate plays an important role). As billed, \eqref{e:OT''} therefore provides a natural unification of the sharp estimates \eqref{e:OT}, \eqref{e:Carneiro} and \eqref{e:PV} of Ozawa--Tsutsumi \cite{OT}, Carneiro \cite{Carneiro} and Planchon--Vega \cite{PV}, respectively.

In addition to the special cases discussed above, the case $\sigma = \frac{4-d}{4}$ is also distinguished since it leads to the kernel $|\zeta - \eta|^2$ on the right-hand side of \eqref{e:OT''} and an additional trick (which we learnt from \cite{Carneiro}) permits the sharp space-time estimates given in the forthcoming Corollary \ref{c:Ctrick}.

A new proof of the Ozawa--Tsutsumi estimate \eqref{e:OT} was given in \cite{BBI}. An advantage of this new proof was that it exposed an underlying heat-flow monotonicity phenomenon. Here, we prove \eqref{e:OT''} following the argument in \cite{BBI} with little extra work. Our main contribution in Theorem \ref{t:OTCunify} then is to unify estimates \eqref{e:OT},
\eqref{e:Carneiro} and \eqref{e:PV} in a natural way, highlight a startling connection to Boltzmann's $H$-theorem, and to establish a full characterisation of extremising initial data for every $\sigma  > \frac{1-d}{4}$. When $\sigma = \frac{2-d}{4}$, it was observed in \cite{OT} that equality holds with $u_0(x) = v_0(x) = \exp(a|x|^2)$ for any $a < 0$, and when $\sigma = 0$, a full characterisation of extremisers was provided in \cite{Carneiro}. 

The lower bound $\sigma > \frac{1-d}{4}$ is necessary; in particular, the optimal constant blows up at this threshold. For $\sigma \in (\frac{1-d}{4},\frac{2-d}{4})$ (so that, in particular, the exponent $4\sigma + d - 2$ on the kernel in \eqref{e:OT''} is negative) and $p,q \in (2,\infty)$ such that
$\frac{1}{p} + \frac{1}{q} = \frac{4\sigma + 3d - 2}{2d}$, it follows from the (forward) Hardy--Littlewood--Sobolev inequality that
\[
\mathcal{F}L^p(\mathbb{R}^d) \times \mathcal{F}L^q(\mathbb{R}^d) \subseteq \Upsilon_\sigma\,,
\]
where $\mathcal{F}L^p$ denotes the Fourier--Lebesgue space of measurable functions whose Fourier transform belongs to $L^p$; such spaces also capture smoothness by the correspondence between decay of the Fourier transform and smoothness. Thus, for such $\sigma, p$ and $q$, we obtain the estimates
\[
\|(-\Delta)^{\sigma}(u\overline{v}) \|_{L^2}^2 \leq C_{d,\sigma} \|u_0\|_{\mathcal{F}L^p}^2 \|v_0\|_{\mathcal{F}L^q}^2.
\]
with a finite, but not necessarily optimal, constant $C_{d,\sigma}$. We also remark that for such $\sigma$, via the Parseval identity, the quantity $I_\sigma(u_0,v_0)$ is given by
\[
I_\sigma(u_0,v_0) = C_{d,\sigma} \int_{\mathbb{R}^d} \frac{\widehat{\mathrm{d}\mu}(x)\widehat{\mathrm{d}\nu}(x)}{|x|^{4\sigma + 2d - 2}} \, \mathrm{d}x
\]
and is the mutual $(4\sigma + d - 2)$-dimensional energy of the measures $\mathrm{d}\mu(\eta) = |\widehat{u}_0(\eta)|^2 \mathrm{d}\eta$ and $\mathrm{d}\nu(\eta) = |\widehat{v}_0(\eta)|^2 \mathrm{d}\eta$.

When $\sigma = \frac{2-d}{4}$, clearly we have $\Upsilon_0 = L^2(\mathbb{R}^d) \times L^2(\mathbb{R}^d)$. Also, for $\sigma > \frac{2-d}{4}$, we can use the trivial upper bound
\[
I_\sigma(u_0,v_0) \lesssim \|u_0\|_{\dot{H}^{2\sigma + \frac{d-2}{2}}}^2\|v_0\|_{L^2}^2 + \|u_0\|_{L^2}^2\|v_0\|_{\dot{H}^{2\sigma + \frac{d-2}{2}}}^2
\]
%along with the reverse Hardy--Littlewood--Sobolev inequality 
so that
\[
H^{2\sigma + \frac{d-2}{2}}(\mathbb{R}^d) \times H^{2\sigma + \frac{d-2}{2}}(\mathbb{R}^d) \subseteq \Upsilon_\sigma 
%\subseteq \mathcal{F}L^p(\mathbb{R}^d) \times \mathcal{F}L^q(\mathbb{R}^d)\,,
\]
%whenever $p,q \in (0,2)$ are such that $\frac{1}{p} + \frac{1}{q} = \frac{4\sigma + 3d - 2}{2d}$. 
where, as usual, $H^s$ denotes the inhomogeneous Sobolev space $L^2 \cap \dot{H}^s$. In the special case $\sigma = \frac{4-d}{4}$ we can be more accurate and obtain the following sharp estimates of this type.
\begin{corollary} \label{c:Ctrick}
  Let $d \geq 2$. Then
  \begin{equation}\label{e:Ctrick}
\|(-\Delta)^{\frac{4-d}{4}}(|u|^2) \|_{L^2}^2 \leq \frac{2^{-d}\pi^{\frac{2-d}{2}}}{\Gamma(\tfrac{d+2}{2})} \|u_0\|_{\dot{H}^1}^2\|u_0\|_{L^2}^2
\end{equation}
for solutions $u$ of \eqref{e:Schro} with initial data $u_0 \in H^1$, and the constant is optimal. Furthermore, the initial data $u_0$ is an extremiser if and only if
$\widehat{u}_0 \in \mathfrak{G}_r$.
\end{corollary}
Note that the class of extremisers is slightly smaller than in Theorem \ref{t:OTCunify}. In the case $d=4$, Corollary \ref{c:Ctrick} was proved by Carneiro \cite{Carneiro} and our result generalises this to $d \geq 2$. We remark that the case $d=2$ involves only classical derivatives, with the estimate \eqref{e:Ctrick} simplifying to
\begin{equation*}
\|\nabla(|u|^2) \|_{L^2(\mathbb{R}^{1+2})} \leq \frac{1}{2} \|u_0\|_{L^2}\|\nabla u_0\|_{L^2}
\end{equation*}
for any $u_0 \in H^1(\mathbb{R}^2)$, where the constant is optimal and attained precisely when $\widehat{u}_0 \in \mathfrak{G}_r$.

It is natural to wonder what happens when we consider $u^2$ rather $|u|^2$, or more generally, $uv$ rather than $u\overline{v}$. Taking the classical exponent $\sigma = \frac{2-d}{4}$ in the Ozawa--Tsutsumi estimates \eqref{e:OT}, we know that
\begin{equation} \label{e:OTnobar}
\|(-\Delta)^{\frac{2-d}{4}}(uv) \|_{L^2}^2 \leq C_{d} \|u_0\|_{L^2}^2 \|v_0\|_{L^2}^2
\end{equation}
holds for some finite constant $C_d$, $d \geq 2$, independent of the initial data $(u_0,v_0) \in L^2(\mathbb{R}^d) \times L^2(\mathbb{R}^d)$. This can easily be seen using Sobolev embedding, H\"older's inequality, and the mixed-norm linear Strichartz estimate $$L^2(\mathbb{R}^d) \to L^4_tL^{\frac{2d}{d-1}}_x(\mathbb{R} \times \mathbb{R}^{d})$$ for the solution of \eqref{e:Schro}.
However, our next result confirms that there is a distinction between \eqref{e:OT} and \eqref{e:OTnobar} at the level of \emph{sharp} estimates whenever $d \geq 3$.
\begin{theorem} \label{t:nobars}
  Suppose $d \geq 3$. Whenever $\widehat{u}_0 \in \mathfrak{G}$ and $v_0$ is a scalar multiple of $u_0$, then $(u_0,v_0)$
  is not a critical point for the functional
  \begin{equation} \label{e:Sfunctional}
  (u_0,v_0) \mapsto \frac{\|(-\Delta)^{\frac{2-d}{4}}(uv) \|_{L^2}}{\|u_0\|_{L^2} \|v_0\|_{L^2}}\,.
  \end{equation}
\end{theorem}
Instead the natural analogues of the estimates in Theorem \ref{t:OTCunify} and Corollary \ref{c:Ctrick} which preserve the class of Gaussian extremisers arise by replacing powers of $-\Delta$ with powers of $|\mathfrak{D}|$, where $\mathfrak{D} = i\partial_t + \frac{1}{2}\Delta$. Further evidence is provided by considering the case of $d=1$; we have already observed that \eqref{e:OT} is really the identity in \eqref{e:OTidentity} in one spatial dimension, and similar considerations in case of $uv$ lead us to an identity involving $\mathfrak{D}$; we expound this point in Section \ref{section:further}.

Writing $|\mathfrak{D}|^\beta$ for the Fourier multiplier operator given by $|\tau + \frac{1}{2}|\xi|^2|^\beta$, and recalling the constant $\mathbf{C}(d)$ in \eqref{e:Carneiroconstant}, we obtain the following.
\begin{theorem} \label{t:Bourgainspaceversion}
Let $d \geq 2$ and $\beta > \frac{1-d}{2}$. Then
\begin{equation} \label{e:Bourgainversion}
||\, |\mathfrak{D}|^\beta(uv)\|_{L^2}^2 \leq  2^{-2\beta} \mathbf{C}(d) \int_{\R^{2d}} |\widehat{u}_0(\zeta)|^2|\widehat{v}_0(\eta)|^2 |\zeta - \eta|^{4\beta+d-2} \, \mathrm{d}\zeta \mathrm{d}\eta
\end{equation}
for solutions $u$ and $v$ of \eqref{e:Schro} with initial data $(u_0,v_0) \in \Upsilon_\beta$. The constant is optimal and is attained if and only if
$\widehat{u}_0 \in \mathfrak{G}$ and $v_0$ is a scalar multiple of $u_0$.
\end{theorem}
\begin{corollary} \label{c:CtrickBspace}
Let $d \geq 2$. Then
\begin{equation} \label{e:CBourgainversion}
\| \, |\mathfrak{D}|^{\frac{2-d}{4}} (u^2) \|_{L^2}^2 \leq  \frac{2^{\frac{2-3d}{2}}\pi^{\frac{2-d}{2}} }{\Gamma(\frac{d}{2})}  \|u_0\|_{L^2}^4
\end{equation}
and
\begin{equation} \label{e:CtrickBourgainversion}
\| \, |\mathfrak{D}|^{\frac{4-d}{4}} (u^2) \|_{L^2}^2 \leq  \frac{2^{\frac{2-3d}{2}}\pi^{\frac{2-d}{2}} }{\Gamma(\frac{d}{2})}  \|u_0\|_{\dot{H}^1}^2\|u_0\|_{L^2}^2
\end{equation}
for solutions $u$ of \eqref{e:Schro} with initial data $u_0 \in L^2$ and $u_0 \in H^1$, respectively, and in each case, the constant is optimal. Initial data $u_0$ is an extremiser for \eqref{e:CBourgainversion} if and only if $\widehat{u}_0 \in \mathfrak{G}$, and $u_0$ is an extremiser for \eqref{e:CtrickBourgainversion} if and only if $\widehat{u}_0 \in \mathfrak{G}_r$.
\end{corollary}
It is interesting to contrast our observations with the case of the half-wave propagator $e^{it\sqrt{-\Delta}}$ where the situation is somewhat different. Sharp space-time estimates which are analogous to those in Theorems \ref{t:OTCunify} and \ref{t:Bourgainspaceversion} have very recently been obtained in \cite{BJO}, in which case the class of extremisers is the same for both $u\overline{v}$ and $uv$ and in each case the multiplier operator is a power of $|\square| = |\partial_t^2 - \Delta|$.

It is possible to prove \eqref{e:OT''} by modifying to the approach of Ozawa--Tsutsumi in \cite{OT}, and similarly, one can prove \eqref{e:Bourgainversion} by appropriately modifying the approach of Foschi in \cite{Foschi}; these approaches are rather different. Here, our proofs of \eqref{e:OT''} and \eqref{e:Bourgainversion} are based on the alternative perspective in \cite{BBI}, which has the main advantage of being simultaneously applicable to \eqref{e:OT''} and \eqref{e:Bourgainversion}, thus permitting a streamlined presentation. A consequence of this is that the characterisation of extremisers in both Theorems \ref{t:OTCunify} and \ref{t:Bourgainspaceversion} may be reduced immediately to finding the solution of the same functional equation. Furthermore, by using the approach based on \cite{BBI} we are able to expose underlying heat-flow monotonicity phenomena in the general context of \eqref{e:OT''} and \eqref{e:Bourgainversion}, extending some of the results in \cite{BBI}. In particular, we shall prove the following.
\begin{theorem} \label{t:flow}
Suppose $d \geq 2$. For any $\sigma > \frac{1-d}{4}$ and initial data $(u_0,v_0) \in \Upsilon_\sigma$, the quantity
\begin{equation*}
\rho \mapsto \mathbf{OT}(d,\sigma) I_\sigma(e^{\rho\Delta}u_0,e^{\rho\Delta}v_0) - \|(-\Delta)^{\sigma}(e^{\rho\Delta}u\,\overline{e^{\rho\Delta}v}) \|_{L^2_{t,x}}^2
\end{equation*}
is nonincreasing on $(0,\infty)$. Similarly, for any $\beta > \frac{1-d}{2}$ and $(u_0,v_0) \in \Upsilon_\beta$, the quantity
\begin{equation*}
\rho\mapsto 2^{-2\beta} \mathbf{C}(d) I_\beta(e^{\rho\Delta}u_0,e^{\rho\Delta}v_0) - \|\, |\mathfrak{D}|^\beta (e^{\rho\Delta}u \,e^{\rho\Delta}v) \|_{L^2_{t,x}}^2
\end{equation*}
is nonincreasing on $(0,\infty)$.
\end{theorem}
The monotonicity of the quantities in Theorem \ref{t:flow} for $\rho > 0$ recovers the sharp estimates in \eqref{e:OT''} and \eqref{e:Bourgainversion}, respectively, and this is seen by comparing the limiting behaviour of the quantities as $\rho \to 0+$ and $\rho \to \infty$. Thus, the functionals are interpolating between the two sides of the inequalities and the heat-flow is evolving arbitrary initial data to a Gaussian shape for large (heat-flow) time $\rho$; we refer the reader to \cite{BBI} for background and further details on this perspective.

\textit{Organisation.} In the next section we prove the sharp estimates appearing in Theorems \ref{t:OTCunify} and \ref{t:Bourgainspaceversion} and Corollaries \ref{c:Ctrick} and \ref{c:CtrickBspace}, along with the heat-flow monotonicity in Theorem \ref{t:flow}. The statements concerning characterisations of extremisers in these results are proved in Section \ref{section:ex}. Finally, in Section \ref{section:further} we discuss the case $d=1$ in order to clarify the relationship between Theorems \ref{t:OTCunify} and \ref{t:Bourgainspaceversion}, and we provide a proof of Theorem \ref{t:nobars}.

\section{Proof of the sharp estimates \eqref{e:OT''}--\eqref{e:CtrickBourgainversion}} \label{section:estimate}

For appropriate functions $F$ on $\mathbb{R}^{1+d}$, we will use the notation $\widetilde{F}$ for the space-time Fourier transform of $F$ given by
\[
\widetilde{F}(\tau,\xi) = \int_{\mathbb{R}^{1+d}} F(t,x) \exp(-i(t\tau + x \cdot \xi))  \, \mathrm{d}t \mathrm{d}x\,.
\]

\begin{proof}[Proof of \eqref{e:OT''}]
Since the argument for $\sigma = \frac{2-d}{4}$ may be found in \cite{BBI} and we only need make straightforward modifications to handle general $\sigma$, we shall be brief in certain parts of the argument.

An application of Plancherel's theorem in space-time gives
\begin{equation*}
\| (-\Delta)^{\sigma} (u \overline{v}) \|_{L^2}^2 = \frac{1}{(2\pi)^{d+1}} \int_{\mathbb{R}^{d+1}} |\xi|^{4\sigma} |\widetilde{(u\overline{v})}(\tau,\xi)|^2 \, \mathrm{d}\xi \mathrm{d}\tau
\end{equation*}
and since $\widetilde{u\overline{v}}  = \frac{1}{(2\pi)^{d+1}} \,\widetilde{u} * \widetilde{\overline{v}}$ we obtain
\begin{align*}
&  \| (-\Delta)^{\sigma} (u \overline{v}) \|_{L^2}^2 \\
& =  \frac{1}{(2\pi)^{3d-1}} \int_{\mathbb{R}^{2d}} \int_{\mathbb{R}^{2d}} |\zeta_1 + \zeta_2|^{4\sigma} \widehat{u}_0(\zeta_1) \widehat{\overline{v}}_0(\zeta_2)  \overline{\widehat{u}_0(\eta_1) \widehat{\overline{v}}_0(\eta_2) }  \,\, \times \\
& \qquad  \qquad \delta(-|\zeta_1|^2 + |\zeta_2|^2 + |\eta_1|^2-|\eta_2|^2) \delta(\zeta_1 + \zeta_2 - \eta_1 -\eta_2) \,\mathrm{d}\zeta\mathrm{d}\eta\,.
\end{align*}
Relabelling the variables $(\zeta_1,\eta_1,\zeta_2,\eta_2) \to (\zeta_1,\eta_1,\eta_2,\zeta_2)$, we have
\begin{align} \label{e:preAMGM}
\| (-\Delta)^{\sigma} (u \overline{v}) \|_{L^2}^2 = \frac{1}{(2\pi)^{3d-1}} \int_{\mathbb{R}^{2d}} \int_{\mathbb{R}^{2d}}   \widehat{U}_0(\zeta) \overline{\widehat{U}_0(\eta)} \, \mathrm{d}\Sigma_\zeta(\eta) \mathrm{d}\zeta
\end{align}
where $U_0 = u_0 \otimes v_0(- \,\cdot)$ and the measure $\mathrm{d}\Sigma_\zeta(\eta)$ is given by
\begin{equation} \label{e:measureforOTversion}
\mathrm{d}\Sigma_\zeta(\eta)  =  |\zeta_1 + \eta_2|^{4\sigma} \delta(|\eta_1|^2 + |\eta_2|^2-|\zeta_1|^2-|\zeta_2|^2) \delta(\eta_1  -\eta_2 - (\zeta_1 - \zeta_2))  \mathrm{d}\eta\,.
\end{equation}
\begin{lemma} \label{l:OTviaFoschi}
For each $\zeta \in \mathbb{R}^{2d}$ we have
\begin{equation*}
\int_{\mathbb{R}^{2d}} \mathrm{d}\Sigma_\zeta  = \pi^{\frac{d-1}{2}} \frac{\Gamma(2\sigma + \frac{d-1}{2})}{2\Gamma(2\sigma + d-1)} |\zeta_1 + \zeta_2|^{4\sigma + d-2}\,.
\end{equation*}
\end{lemma}
\begin{proof}
We have
\begin{align*}
\int_{\mathbb{R}^{2d}} \mathrm{d}\Sigma_\zeta(\eta)  & = \frac{1}{2} \int_{\mathbb{R}^{d}}  |\xi_2|^{4\sigma} \delta(|\xi_2|^2 - \xi_2 \cdot (\zeta_1 + \zeta_2)) \mathrm{d}\xi_2 \\
& = \frac{1}{2} \int_{\mathbb{S}^{d-1}} \int_0^\infty  r^{4\sigma + d-2} \delta(r - \omega \cdot (\zeta_1 + \zeta_2)) \mathrm{d}r \mathrm{d}\omega
\end{align*}
via the change of variables $(\xi_1,\xi_2) = (\eta_1+\zeta_2,\eta_2+\zeta_1)$ and subsequently polar coordinates $\xi_2 = r \omega$. By applying a rotation, we may replace $\zeta_1 + \zeta_2$ with $|\zeta_1 + \zeta_2| \mathbf{e}_1$, and thus (via, for example, the Funk--Hecke formula; see \cite{AtkinsonHan})
\begin{align*}
\int_{\mathbb{R}^{2d}} \mathrm{d}\Sigma_\zeta(\eta)  = \frac{\pi^\frac{d-1}{2}}{\Gamma(\frac{d-1}{2})}|\zeta_1 + \zeta_2|^{4\sigma+d-2} \int_{0}^1  s^{4\sigma + d-2} (1-s^2)^{\frac{d-3}{2}} \mathrm{d}s\,.
\end{align*}
To obtain the claimed expression for the constant we change variables once more
\begin{align*}
\int_{0}^1  s^{4\sigma + d-2} (1-s^2)^{\frac{d-3}{2}} \mathrm{d}s = \frac{1}{2} \int_{0}^1  t^{2\sigma + \frac{d-3}{2}} (1-t)^{\frac{d-3}{2}} \mathrm{d}t = \frac{1}{2} \mathrm{B}(\tfrac{d-1}{2},2\sigma + \tfrac{d-1}{2})\,,
\end{align*}
where $\mathrm{B}$ is the beta function. An application of the identity $\mathrm{B}(x,y) = \frac{\Gamma(x)\Gamma(y)}{\Gamma(x+y)}$
completes the proof.
\end{proof}
Lemma \ref{l:OTviaFoschi} and the symmetry relation $\mathrm{d}\Sigma_\eta(\zeta)\mathrm{d}\eta = \mathrm{d}\Sigma_\zeta(\eta)\mathrm{d}\zeta$ imply that
\begin{align*}
\mathbf{OT}(d,\sigma) I_\sigma(u_0,v_0) & = \frac{1}{(2\pi)^{3d-1}} \int_{\mathbb{R}^{2d}}\int_{\mathbb{R}^{2d}} |\widehat{U}_0(\zeta)|^2   \, \mathrm{d}\Sigma_\zeta(\eta) \mathrm{d}\zeta \\
& = \frac{1}{2(2\pi)^{3d-1}} \int_{\mathbb{R}^{2d}}\int_{\mathbb{R}^{2d}}\big( |\widehat{U}_0(\zeta)|^2 + |\widehat{U}_0(\eta)|^2 \big)   \, \mathrm{d}\Sigma_\zeta(\eta) \mathrm{d}\zeta\,.
\end{align*}
Since the left-hand side of \eqref{e:preAMGM} is nonnegative, we may take the real part of both sides and apply the arithmetic--geometric mean inequality
\[
2\mbox{Re} \,\big(\widehat{U}_0(\zeta) \overline{\widehat{U}_0(\eta)} \big)\leq |\widehat{U}_0(\zeta)|^2 + |\widehat{U}_0(\eta)|^2
\]
to obtain
\[
\| (-\Delta)^{\sigma} (u \overline{v}) \|_{L^2}^2 \leq \mathbf{OT}(d,\sigma) I_\sigma(u_0,v_0)
\]
which establishes \eqref{e:OT''}.
\end{proof}

\begin{proof}[Proof of \eqref{e:Bourgainversion}]
Writing $\widetilde{uv} = \frac{1}{(2\pi)^{d+1}} \,\widetilde{u} * \widetilde{v}$ leads to
\begin{align*}
& ||\, |\mathfrak{D}|^\beta(uv)\|_{L^2}^2 \\
& =  2^{-3d+1-2\beta} \pi^{-1-3d} \int_{\mathbb{R}^{2d}} \int_{\mathbb{R}^{2d}} |\zeta_1 - \zeta_2|^{4\beta} \widehat{u}_0(\zeta_1) \widehat{v}_0(\zeta_2) \overline{\widehat{u}_0(\eta_1) \widehat{v}_0(\eta_2)}   \,\, \times \\
& \qquad  \qquad \delta(|\zeta_1|^2 + |\zeta_2|^2 - |\eta_1|^2-|\eta_2|^2) \delta(\zeta_1 + \zeta_2 - \eta_1 -\eta_2) \,\mathrm{d}\zeta\mathrm{d}\eta \\
& = 2^{-3d+1-2\beta} \pi^{-1-3d} \int_{\mathbb{R}^{2d}} \int_{\mathbb{R}^{2d}}   \widehat{U}_0(\zeta) \overline{\widehat{U}_0(\eta)} \, \mathrm{d}\Sigma_\zeta(\eta) \mathrm{d}\zeta \,.
\end{align*}
Here, $U_0 = u_0 \otimes v_0$, the measure $\mathrm{d}\Sigma_\zeta(\eta)$ is given by
\begin{equation} \label{e:measureforBversion}
\mathrm{d}\Sigma_\zeta(\eta)  =  |\zeta_1 - \zeta_2|^{4\beta} \delta(|\zeta_1|^2 + |\zeta_2|^2-|\eta_1|^2-|\eta_2|^2) \delta(\zeta_1  + \zeta_2 - \eta_1 - \eta_2)  \mathrm{d}\eta
\end{equation}
and where we have used the fact that if $\tau = -|\zeta_1|^2-|\zeta_2|^2$ and $\xi = \zeta_1 + \zeta_2$, then
\[
|\tau + \tfrac{1}{2}|\xi|^2| = \tfrac{1}{2}|\zeta_1 - \zeta_2|^2.
\]
\begin{remark}
Notice that the function $U_0$ and the measure $\mathrm{d}\Sigma_\zeta$ in the current proof of \eqref{e:Bourgainversion} are slightly different to the $U_0$ and $\mathrm{d}\Sigma_\zeta$ used in the previous proof of \eqref{e:OT''}. We have decided to use the same notation in order to highlight that the two proofs are structurally the same.
\end{remark}
\begin{lemma} \label{l:convitself}
For each $\zeta \in \mathbb{R}^{2d}$ we have
\[
\int_{\mathbb{R}^{2d}} \mathrm{d}\Sigma_\zeta  = \frac{\pi^{\frac{d}{2}}}{2^{d-1}\Gamma(\frac{d}{2})}  |\zeta_1 - \zeta_2|^{4\beta+d-2}.
\]
\end{lemma}
\begin{proof}
Using the change of variables $(\xi_1,\xi_2) = (\tfrac{1}{2}(\zeta_1+\zeta_2) - \eta_1,\tfrac{1}{2}(\zeta_1+\zeta_2) - \eta_2)$  and a subsequent polar coordinate change of variables in $\xi_2$, we have
\begin{align*}
\int_{\mathbb{R}^{2d}} \mathrm{d}\Sigma_\zeta(\eta) & = |\zeta_1-\zeta_2|^{4\beta} \int_{\mathbb{R}^{2d}} \delta(\tfrac{1}{2}|\zeta_1 - \zeta_2|^2 - |\xi_1|^2 - |\xi_2|^2) \delta(\xi_1 + \xi_2) \, \mathrm{d}\xi \\
& = |\mathbb{S}^{d-1}| |\zeta_1-\zeta_2|^{4\beta} \int_{0}^\infty \delta(\tfrac{1}{2}|\zeta_1 - \zeta_2|^2 - 2r^2) r^{d-1} \, \mathrm{d}r \\
& = \frac{\pi^{\frac{d}{2}}}{2^{d-1}\Gamma(\frac{d}{2})}  |\zeta_1 - \zeta_2|^{4\beta+d-2}.
\end{align*}
In the last step, we used the well-known formula $|\mathbb{S}^{d-1}| = \frac{2\pi^{\frac{d}{2}}}{\Gamma(\frac{d}{2})}$ for the measure of the unit sphere in $\mathbb{R}^d$.
\end{proof}
As in the proof of \eqref{e:OT''}, we now use the symmetry relation $\mathrm{d}\Sigma_\eta(\zeta)\mathrm{d}\eta = \mathrm{d}\Sigma_\zeta(\eta)\mathrm{d}\zeta$, Lemma \ref{l:convitself} and the arithmetic--geometric mean inequality to obtain
\begin{align*}
||\, |\mathfrak{D}|^\beta(uv)\|_{L^2}^2 & =   2^{-3d+1-2\beta} \pi^{-1-3d} \int_{\mathbb{R}^{2d}} \int_{\mathbb{R}^{2d}}   \widehat{U}_0(\zeta) \overline{\widehat{U}_0(\eta)} \, \mathrm{d}\Sigma_\zeta(\eta) \mathrm{d}\zeta \\
& \leq   2^{-3d-2\beta} \pi^{-1-3d} \int_{\mathbb{R}^{2d}}\int_{\mathbb{R}^{2d}}\big( |\widehat{U}_0(\zeta)|^2 + |\widehat{U}_0(\eta)|^2 \big)   \, \mathrm{d}\Sigma_\zeta(\eta) \mathrm{d}\zeta\ \\
& = 2^{-2\beta}\mathbf{C}(d)  I_\beta(u_0,v_0)
\end{align*}
as desired.
\end{proof}

\begin{proof}[Proofs of \eqref{e:Ctrick}, \eqref{e:CBourgainversion} and \eqref{e:CtrickBourgainversion}]
The estimate \eqref{e:CBourgainversion} is an immediate consequence of \eqref{e:Bourgainversion} and Plancherel's theorem.

For \eqref{e:Ctrick} and \eqref{e:CtrickBourgainversion}, expanding $|\zeta - \eta|^2$ and using Plancherel's theorem we obtain
\begin{align*}
I_{\frac{4-d}{4}}(u_0,u_0) = 2(2\pi)^{2d} \|u_0\|_{L^2}^2 \|u_0\|_{\dot{H}^1}^2 - 2 \int_{\mathbb{R}^{2d}}  |\widehat{u}_0(\zeta)|^2|\widehat{u}_0(\eta)|^2 \zeta\cdot\eta \, \mathrm{d}\zeta \mathrm{d}\eta
\end{align*}
and therefore
\begin{equation} \label{e:Ctrickconseq}
I_{\frac{4-d}{4}}(u_0,u_0) \leq 2(2\pi)^{2d} \|u_0\|_{L^2}^2 \|u_0\|_{\dot{H}^1}^2
\end{equation}
for any $u_0 \in H^1$. The estimates \eqref{e:Ctrick} and \eqref{e:CtrickBourgainversion} now follow at once from \eqref{e:OT''} and \eqref{e:Bourgainversion}.
\end{proof}

\begin{proof}[Proof of Theorem \ref{t:flow}]
The above proof of \eqref{e:OT''} in fact shows that
\begin{align*}
\mathbf{OT}(d,\sigma) I_\sigma(u_0,v_0) - \|(-\Delta)^{\sigma}(u\overline{v})\|_{L^2}^2 = c \int_{\mathbb{R}^{2d}}\int_{\mathbb{R}^{2d}}\big|\widehat{U}_0(\zeta) - \widehat{U}_0(\eta)\big|^2   \, \mathrm{d}\Sigma_\zeta(\eta) \mathrm{d}\zeta
\end{align*}
where $\frac{1}{c} = 2(2\pi)^{3d-1}$, $U_0 = u_0 \otimes v_0(- \,\cdot)$ and the measure $\mathrm{d}\Sigma_\zeta(\eta)$ is given by
\eqref{e:measureforOTversion}. Replacing $(u_0,v_0)$ with $(e^{\rho\Delta}u_0,e^{\rho\Delta}v_0)$ for fixed $\rho > 0$, commuting the Schr\"odinger and heat flows, and using the support of $\mathrm{d}\Sigma_\zeta$, we obtain
\begin{align*}
& \mathbf{OT}(d,\sigma) I_\sigma(e^{\rho\Delta}u_0,e^{\rho\Delta}v_0) -\|(-\Delta)^{\sigma}(e^{\rho\Delta}u \, \overline{e^{\rho\Delta}v})\|_{L^2_{t,x}}^2 \\
& \qquad \qquad = c \int_{\mathbb{R}^{2d}}\int_{\mathbb{R}^{2d}} e^{-2\rho(|\zeta_1|^2 + |\zeta_2|^2)} \big| \widehat{U}_0(\zeta) - \widehat{U}_0(\eta)\big|^2   \, \mathrm{d}\Sigma_\zeta(\eta) \mathrm{d}\zeta
\end{align*}
which is manifestly nonincreasing for $\rho \in (0,\infty)$.

A similar argument based on the previous proof of \eqref{e:Bourgainversion} shows that
\begin{align*}
& 2^{-2\beta} \mathbf{C}(d) I_\beta(e^{\rho\Delta}u_0,e^{\rho\Delta}v_0) - \|\, |\mathfrak{D}|^\beta (e^{\rho\Delta}u \,e^{\rho\Delta}v) \|_{L^2_{t,x}}^2 \\
& \qquad \qquad = c \int_{\mathbb{R}^{2d}}\int_{\mathbb{R}^{2d}} e^{-2\rho(|\zeta_1|^2 + |\zeta_2|^2)} \big| \widehat{U}_0(\zeta) - \widehat{U}_0(\eta)\big|^2   \, \mathrm{d}\Sigma_\zeta(\eta) \mathrm{d}\zeta
\end{align*}
where, now, $\frac{1}{c} = 2^{3d+2\beta} \pi^{1+3d} $, $U_0 = u_0 \otimes v_0$ and the measure $\mathrm{d}\Sigma_\zeta(\eta)$ is given by \eqref{e:measureforBversion}. This completes our proof of Theorem \ref{t:flow}.
\end{proof}

\begin{remark}
It is clear from the proof of Theorem \ref{t:flow} that the monotone quantities are in fact \textit{completely monotone} since their $\rho$-derivatives have sign $(-1)^j$ for every $j \in \mathbb{N}$.
\end{remark}

\section{Characterisation of extremising initial data} \label{section:ex}
It was shown in Section \ref{section:estimate} that \eqref{e:OT''} and \eqref{e:Bourgainversion} follow from a single application of the arithmetic--geometric mean inequality
\begin{equation} \label{e:AMGM}
2\mbox{Re} \,\big(\widehat{U}_0(\zeta) \overline{\widehat{U}_0(\eta)} \big)\leq |\widehat{U}_0(\zeta)|^2 + |\widehat{U}_0(\eta)|^2
\end{equation}
for each $\zeta \in \mathbb{R}^{2d}$ and each $\eta$ in the support of $\mathrm{d}\Sigma_\zeta$, which is obviously an equality if and only if $\widehat{U}_0(\zeta)$ and $\widehat{U}_0(\eta)$ coincide. For each estimate \eqref{e:OT''} and \eqref{e:Bourgainversion}, $U_0$ and $\mathrm{d}\Sigma_\zeta$ are slightly different.

For \eqref{e:Bourgainversion}, $U_0 = u_0 \otimes v_0$ and $\mathrm{d}\Sigma_\zeta$ is given by \eqref{e:measureforBversion}, which means $(u_0,v_0)$ is an extremising pair of initial data if and only if
\begin{equation*}
\widehat{u}_0(\zeta_1) \widehat{v}_0(\zeta_2) = \widehat{u}_0(\eta_1) \widehat{v}_0(\eta_2)
\end{equation*}
for almost every $\zeta \in \mathbb{R}^{2d}$ and almost every $\eta \in \mathbb{R}^{2d}$ satisfying $|\eta_1|^2 + |\eta_2|^2 = |\zeta_1|^2 + |\zeta_2|^2$ and $\eta_1 + \eta_2 = \zeta_1 + \zeta_2$, or equivalently
\begin{equation} \label{e:feqnFoschigen}
\widehat{u}_0(\zeta_1) \widehat{v}_0(\zeta_2) = \Lambda(|\zeta_1|^2+|\zeta_2|^2,\zeta_1+\zeta_2)
\end{equation}
for almost every $\zeta \in \mathbb{R}^{2d}$, and where $\Lambda$ is a scalar function. 

For \eqref{e:OT''}, $U_0 = u_0 \otimes v_0(-\, \cdot)$ and $\mathrm{d}\Sigma_\zeta$ is given by \eqref{e:measureforOTversion}. Since
$\widehat{U}_0(\zeta) = \widehat{u}_0(\zeta_1)\widehat{v}_0(-\zeta_2)$ and for $\eta$ in the support of $\mathrm{d}\Sigma_\zeta$ we have $|\eta_1|^2 + |\eta_2|^2 = |\zeta_1|^2 + |\zeta_2|^2$ and $\eta_1 - \eta_2 = \zeta_1 - \zeta_2$, it follows that $(u_0,v_0)$ is an extremising pair of initial data for \eqref{e:OT''} if and only if \eqref{e:feqnFoschigen} holds.

%If $\widehat{u}_0 \in \mathfrak{G}$ and $\widehat{v}_0$ is a scalar multiple of $\widehat{u}_0$, then it is trivial to see that \eqref{e:feqnFoschigen} holds. Showing that extremising initial data $(u_0,v_0)$ necessarily have this form is a non-trivial task. 

Using the above observations, the characterisation of extremisers for \eqref{e:OT''} and \eqref{e:Bourgainversion} will be established if we can show that, whenever $I_\sigma(u_0,v_0) < \infty$, the pair $(u_0,v_0)$ solves \eqref{e:feqnFoschigen} if and only if $\widehat{u}_0 \in \mathfrak{G}$ and $v_0$ is a scalar multiple of $u_0$. The sufficiency part of this claim is obvious so we show how to justify the necessity part.

\begin{remark} The functional equation 
\begin{equation*} 
F(\zeta_1) F(\zeta_2) = \Lambda(|\zeta_1|^2+|\zeta_2|^2,\zeta_1+\zeta_2)
\end{equation*}
is known in the kinetic equations literature as the Maxwell--Boltzmann (MB) functional equation, and the system of equations
\[
\zeta_1' + \zeta_2' = \zeta_1 + \zeta_2 \qquad \text{and} \qquad |\zeta_1'|^2 + |\zeta_2'|^2 = |\zeta_1|^2 + |\zeta_2|^2
\]
express the conservation of momentum and kinetic energy, respectively, during a binary collision, where $(\zeta_1,\zeta_2)$ are the velocities of a pair of particles before collision, and $(\zeta_1',\zeta_2')$ are the velocities of the same pair after collision. 

It is known that if $F \in L^1(\mathbb{R}^d)$, then $F$ satisfies the MB equation if and only if $F \in \mathfrak{G}$. A justification of this can be found in lecture notes of Villani \cite{Villani} (but the result goes back further; see, for example, work of Lions \cite{Lions} and Perthame \cite{Perthame}). We make use of this below to complete the characterisation of extremisers for \eqref{e:OT''} and \eqref{e:Bourgainversion}.
\end{remark}

\begin{proof}[Extremisers characterisation for \eqref{e:OT''} and \eqref{e:Bourgainversion}] \, The right-hand side of \eqref{e:feqnFoschigen} is symmetric in $\zeta_1$ and $\zeta_2$ and therefore $\widehat{u}_0$ and $\widehat{v}_0$ are linearly dependent; by scaling we may assume $u_0 = v_0$. To show that $\widehat{u}_0 \in \mathfrak{G}$, we argue differently depending on sign of the exponent on the kernel in $I_\sigma(u_0,u_0)$.

First we consider the case $\sigma < \frac{2-d}{4}$ and define
\[
F(\zeta) = e^{-|\zeta|^2} \widehat{u}_0(\zeta).
\]
By the Cauchy--Schwarz inequality on $L^2(\mathbb{R}^{2d})$, we obtain
\begin{align*}
\left( \int_{\mathbb{R}^d}  |F| \right)^4 \leq I_\sigma(u_0,u_0) \int_{\mathbb{R}^d} \int_{\mathbb{R}^d} \frac{e^{-2(|\zeta|^2 + |\eta|^2)}}{|\zeta - \eta|^{4\sigma + d - 2}}  \, \mathrm{d}\zeta \mathrm{d}\eta\,
\end{align*}
and the double integral is finite since $4\sigma + d - 2 < 0$. Thus $F \in L^1(\mathbb{R}^d)$ and clearly inherits the property of being a solution of the MB equation from $\widehat{u}_0$. From the above remark, it follows that $F \in \mathfrak{G}$, and therefore $\widehat{u}_0 \in \mathfrak{G}$.

For $\sigma \geq \frac{2-d}{4}$, we argue somewhat differently, and begin by observing that we may apply the reverse Hardy--Littlewood--Sobolev inequality (see, for example, \cite{DZ}) to obtain
\[
I_{\sigma}(u_0,u_0) \geq C_{d,\sigma} \| \widehat{u}_0 \|_{p}^4
\]
for $p = \frac{4d}{4\sigma + 3d - 2} \in (0,1]$. Therefore, $|\widehat{u}_0|^{p} \in L^1(\mathbb{R}^d)$ is a solution of the MB equation; again using the above remark, it follows that $|\widehat{u}_0|^p \in \mathfrak{G}$, and hence $|\widehat{u}_0| \in \mathfrak{G}$. The polar part $|\widehat{u}_0|^{-1} \widehat{u}_0$ is a solution of the MB equation too; whilst this function is not integrable on $\mathbb{R}^d$, we may argue as above to see that $e^{-|\cdot|^2}|\widehat{u}_0|^{-1} \widehat{u}_0 \in \mathfrak{G}$. It follows that $\widehat{u}_0 \in \mathfrak{G}$, as required.
\end{proof}

\begin{remark}
If we were to know that \emph{locally integrable} solutions of the MB equation must be Gaussian, the above proof would be streamlined (by avoiding the introduction of the Gaussian factors). This is, in fact, true and in this extended remark we include an outline of a proof since it contains some interesting features.

In the case $d=2$, in the characterisation of the extremisers for \eqref{e:FHZd=2}, Foschi \cite{Foschi} showed that, for locally integrable functions, the MB equation admits only Gaussian solutions. This was achieved by first showing such solutions must in fact be continuous; more precisely, Foschi showed the existence of smooth maps $P$ and $Q$ such that 
\begin{equation} \label{e:PQ}
\widehat{u}_0(x)\widehat{u}_0(y) = \widehat{u}_0(P(x,y))\widehat{u}_0(Q(x,y))
\end{equation}
and  
$
\det\frac{\partial P}{\partial y}(x,y), \det\frac{\partial Q}{\partial y}(x,y) \neq 0$, from which the desired continuity follows upon integration in one of the variables. 

The mappings $P$ and $Q$ used by Foschi were given by
\begin{align*}
2P(x,y) =  x+y + H(x-y)  \qquad \text{and} \qquad 2Q(x,y) = x+y - H(x-y) \,,
\end{align*}
where $H(x_1,x_2) = (-x_2,x_1)$. Importantly for the argument, $H$ is smooth, isometric and $H(x) \perp x$. This extends to $\mathbb{R}^d$ when $d$ is \textit{even}, by taking $P, Q$ exactly as above, now with the block-form matrix
\[
H(x_1,x_2,\ldots,x_{d-1},x_d) = (-x_2,x_1,-x_4,x_3,\ldots, -x_d,x_{d-1})\,.
\]
Interestingly, it seems we cannot proceed like this when $d$ is \emph{odd} because of the Hairy Ball theorem from algebraic topology. In particular, it follows (see, for example, \cite{Rotman}) from the Hairy Ball theorem that any continuous map $H$ from an even dimensional sphere to itself cannot have the property that $H(x) \perp x$ for every $x$ (because there must exist some point on the sphere which is fixed or sent to its antipode). So, when $d$ is odd, we cannot find an isometric map $H : \mathbb{R}^d \to \mathbb{R}^d$ which is continuous and is such that $H(x) \perp x$.

Despite this obstruction, we remark that construction of $P$ and $Q$ satisfying \eqref{e:PQ} is possible in all dimensions; specifically, we may take $P(x,y)$ and $Q(x,y)$ to be the two intersection points of the sphere in $\mathbb{R}^d$ with centre $\frac{1}{2}(x+y)$ and radius $\frac{1}{2}|x-y|$ and the straight line passing through the origin and the centre of this sphere. 

Once continuity of $\widehat{u}_0$ is established, one can show that solutions of the MB equation never vanish (for example, by extending Lemma 7.13 in \cite{Foschi} to higher dimensions) after which it becomes much easier to obtain that the solutions must be Gaussian. Indeed, after normalising so that $\widehat{u}_0(0) = 1$, one can use \eqref{e:feqnFoschigen} to show that
\[
\mathrm{e}(\zeta) := \log \widehat{u}_0(\zeta) + \log \widehat{u}_0(-\zeta) \qquad \text{and} \qquad \mathrm{o}(\zeta) := \log \widehat{u}_0(\zeta) - \log \widehat{u}_0(-\zeta),
\]
are \emph{orthogonally additive} functions (that is, additive when restricted to orthogonal vectors) which are even and odd, respectively. Since we have established that these functions are continuous, it is possible to show that we must have $\mathrm{e}(\zeta) = a|\zeta|^2$ and $\mathrm{o}(\zeta) = b \cdot \zeta$, for some $a \in \mathbb{C}$ and $b \in \mathbb{C}^d$, and hence $\widehat{u}_0$ is Gaussian.
\end{remark}

\begin{proof}[Extremisers characterisation for \eqref{e:Ctrick} and \eqref{e:CtrickBourgainversion}]
We saw in Section \ref{section:estimate} that the estimates \eqref{e:Ctrick} and \eqref{e:CtrickBourgainversion} follow from \eqref{e:OT''} and \eqref{e:Bourgainversion}, respectively, followed by \eqref{e:Ctrickconseq}. When $u_0 = v_0$, extremisers of \eqref{e:OT''} and \eqref{e:Bourgainversion} are such that $\widehat{u}_0 \in \mathfrak{G}$, and since
\begin{equation} \label{e:square}
\int_{\mathbb{R}^{2d}} |\widehat{u}_0(\zeta)|^2 |\widehat{u}_0(\eta)|^2 \zeta \cdot \eta  \, \mathrm{d}\zeta \mathrm{d}\eta
\end{equation}
vanishes when $|\widehat{u}_0|$ is radial, it is clear that we have equality in \eqref{e:Ctrick} and \eqref{e:CtrickBourgainversion} whenever $\widehat{u}_0 \in \mathfrak{G}_r$.

In order to show that there are no further extremisers, it suffices to show that if 
\[
\widehat{u}_0(\eta) = \exp(a|\eta|^2 + b \cdot \eta + c) 
\]
with $a, c \in \mathbb{C}$, $b \in \mathbb{C}^d$ and $\mbox{Re}(a) < 0$, then the quantity in \eqref{e:square} is nonzero whenever $\text{Re}(b)$ is nonzero. For such $b \in \mathbb{C}^d$ we may perform a change of variables $(\zeta,\eta) \mapsto (R \zeta, R\eta)$ in \eqref{e:square}, for a suitably chosen rotation $R$, so that it suffices to consider $b \in \mathbb{C}^d$ such that $\text{Re}(b) = b_1 \mathbf{e}_1$, where $b_1$ is a strictly positive real number. Now
\begin{align*}
\int_{\mathbb{R}^{2d}} |\widehat{u}_0(\zeta)|^2 |\widehat{u}_0(\eta)|^2 \zeta \cdot \eta  \, \mathrm{d}\zeta \mathrm{d}\eta \geq \left(\int_{\mathbb{R}^d} |\widehat{u}_0(\eta)|^2 \eta_1 \, \mathrm{d}\eta \right)^2
\end{align*}
and for such $u_0$ we have
\begin{align*}
\int_{\mathbb{R}^d} |\widehat{u}_0(\eta)|^2 \eta_1 \, \mathrm{d}\eta & = \exp(2\text{Re}(c)) \int_{\mathbb{R}^d} \exp(2\text{Re}(a)|\eta|^2 + 2b_1\eta_1) \eta_1 \, \mathrm{d}\eta \\
& = C \int_{\mathbb{R}} \exp(2\text{Re}(a)\eta_1^2 + 2b_1\eta_1) \eta_1 \, \mathrm{d}\eta_1 \\
& = C \int_0^\infty \exp(2\text{Re}(a)\eta_1^2 + 2b_1\eta_1) \eta_1(1 - \exp(-2b_1\eta_1) ) \, \mathrm{d}\eta_1\,,
\end{align*}
where $C$ is some strictly positive constant depending on $a$ and $c$. Since $b_1 > 0$ it follows that the quantity in \eqref{e:square} is nonzero, as desired.
\end{proof}

\section{One spatial dimension and the proof of Theorem \ref{t:nobars}} \label{section:further}

\subsection{One spatial dimension and the role of the conjugate} In the case of one spatial dimension, there are identities which are the analogues of the sharp estimates in Theorems \ref{t:OTCunify} and \ref{t:Bourgainspaceversion}. We present these identities briefly here, for completeness and to elucidate the role of the complex conjugation causing the change of Fourier multiplier operator from powers of $-\Delta$ to powers of $\mathfrak{D}$, thus justifying our billing of Theorem \ref{t:Bourgainspaceversion} as a natural analogue of Theorem \ref{t:OTCunify}.

For the analogue of \eqref{e:OT''} when $d=1$, we have
\begin{equation*}
\|(-\partial_x^2)^{\sigma}(u\overline{v}) \|_{L^2(\mathbb{R}^2)}^2 = \frac{1}{2(2\pi)^2} \int_{\R^{2}}|\widehat{u}_0(\zeta)|^2|\widehat{v}_0(\eta)|^2 |\zeta - \eta|^{4\sigma-1} \, \mathrm{d}\zeta \mathrm{d}\eta
\end{equation*}
by the well-known approach of writing
\begin{equation} \label{e:identity1}
(u\overline{v})(t,x) = \frac{1}{(2\pi)^2} \int_{\mathbb{R}^2} \exp(ix(\zeta - \eta))\exp(-it(\zeta^2 - \eta^2)) \widehat{u}_0(\zeta) \overline{\widehat{v}_0(\eta)} \,\mathrm{d}\zeta \mathrm{d}\eta\,,
\end{equation}
changing variables $(\zeta,\eta) \mapsto (\zeta - \eta,\zeta^2 - \eta^2)$, using Plancherel's Theorem, and then undoing the previous change of variables. The jacobian from the change of variables is $2|\zeta - \eta|$ and it is clear from \eqref{e:identity1} that this interacts precisely on taking powers of $\partial_x$-derivatives of $(u\overline{v})(t,x)$.

On the other hand, for the analogue of \eqref{e:Bourgainversion}, we have
\begin{equation*}
\|\, |\mathfrak{D}|^\beta (uv) \|_{L^2(\mathbb{R}^2)}^2 =  \frac{1}{(2^{\beta + 2}\pi)^2}  \int_{\R^{2}} |\widehat{u}_0(\zeta) \widehat{v}_0(\eta)+\widehat{u}_0(\eta)\widehat{v}_0(\zeta)|^2 |\zeta - \eta|^{4\beta-1} \, \mathrm{d}\zeta \mathrm{d}\eta
\end{equation*}
and therefore, if $\widehat{u}_0$ and $\widehat{v}_0$ have separated supports,
\begin{equation*}
\|\, |\mathfrak{D}|^\beta (uv) \|_{L^2(\mathbb{R}^2)}^2 =  \frac{1}{2(2^{\beta + 1}\pi)^2}  \int_{\R^{2}} |\widehat{u}_0(\zeta)|^2 |\widehat{v}_0(\eta)|^2  |\zeta - \eta|^{4\beta-1} \, \mathrm{d}\zeta \mathrm{d}\eta\,.
\end{equation*}
This follows in a similar way by writing
\begin{equation*}
(uv)(t,x) = \frac{1}{(2\pi)^2} \int_{\mathbb{R}^2} \exp(ix(\zeta + \eta))\exp(-it(\zeta^2 + \eta^2)) \widehat{u}_0(\zeta) \widehat{v}_0(\eta) \,\mathrm{d}\zeta \mathrm{d}\eta\,,
\end{equation*}
and conjugating use of Plancherel's Theorem with the change of variables $(\zeta,\eta) \mapsto (\zeta + \eta,-\zeta^2 - \eta^2)$ on the half-plane $\mathbb{H} = \{ (\zeta,\eta) \in \mathbb{R}^2 : \zeta < \eta \}$. The jacobian from the change of variables is again $2|\zeta - \eta|$, so it no longer interacts precisely with $\partial_x$-derivatives. Powers of $|\mathfrak{D}|$ do interact precisely with $uv$ since $-|\tau + \frac{1}{2}\xi^2| = \frac{1}{2}|\zeta - \eta|^2
$, where $(\tau,\xi) = (-\zeta^2-\eta^2,\zeta+\eta)$.

Theorem \ref{t:nobars} further reinforces the point that at the level of sharp estimates, it is natural to change the shape of the Fourier multiplier when considering $u^2$ rather than $|u|^2$. We end with a proof of this result.

\begin{proof}[Proof of Theorem \ref{t:nobars}]
If $\Phi$ is the functional given by
\[
\Phi(u_0,v_0) = \frac{\|(-\Delta)^{\frac{2-d}{4}}(uv) \|_{L^2}}{\|u_0\|_{L^2}\|v_0\|_{L^2}}
\]
then one can show that
\[
\lim_{\varepsilon \to 0} \frac{\Phi(u_0+\varepsilon U_0,v_0+\varepsilon V_0)-\Phi(u_0,v_0)}{\varepsilon}=0
\]
for all $(U_0,V_0) \in L^2(\mathbb{R}^d) \times L^2(\mathbb{R}^d)$ if and only if
\begin{equation*}
\mbox{Re} \int_{\mathbb{R}^{d+1}} M \, \widetilde{uv} \, (\overline{\widetilde{uV} + \widetilde{Uv}})     = \mbox{Re} \big( \Phi(u_0,v_0)  (\|v_0\|_2^2 \langle u_0,U_0 \rangle + \|u_0\|_2^2 \langle v_0,V_0 \rangle )\big)
\end{equation*}
for all $(U_0,V_0) \in L^2(\mathbb{R}^d) \times L^2(\mathbb{R}^d)$, where $M(\tau,\xi) = |\xi|^{2-d}$, and $u, v, U$ and $V$ are the evolutions of $u_0, v_0, U_0$ and $V_0$, respectively, under $e^{it\Delta}$. Therefore, by taking $V_0 = 0$ and complex conjugation on both sides, it follows that if $(u_0,v_0)$ is a critical point then necessarily
\begin{equation*}
\int_{\mathbb R^{d+1}} M(\tau,\xi)  \overline{\widetilde{uv}}(\tau,\xi)\widetilde{Uv}(\tau,\xi)  \,\mathrm{d}\tau\mathrm{d}\xi   =   \Phi(u_0,v_0)  \|v_0\|_2^2 \langle U_0,u_0 \rangle
\end{equation*}
for all $U_0 \in L^2(\mathbb{R}^d)$, and hence
\begin{equation} \label{e:ELnec}
\int_\mathbb{R} \exp(it\Delta) \bigg(  \widetilde{M \overline{\widetilde{uv}}}(t,\cdot) v(t,\cdot)\bigg)(y) \, \mathrm{d}t = \Phi(u_0,v_0)  \|v_0\|_2^2 \overline{u_0(y)}
\end{equation}
for almost every $y \in \mathbb{R}^d$. 

We claim that \eqref{e:ELnec} fails to hold whenever $d \geq 3$, 
$
\widehat{u}_0(\eta) = \exp(a|\eta|^2 + b \cdot \eta + c) \in \mathfrak{G}
$ 
and $v_0$ is a scalar multiple of $u_0$ (from our calculation below it will be apparent that such functions do satisfy \eqref{e:ELnec} when $d=2$). In what follows we denote by $C$ a positive constant, depending only on the parameters $a, b, c$ defining $u_0$, and the ambient dimension $d$. The constant $C$ may change from line to line.

Firstly, since
\[
M(\tau,\xi)\overline{\widetilde{uv}}(\tau,\xi) = C\int_{\R^{2d}}  \frac{e^{\overline{a}(|\eta_1|^2 + |\eta_2|^2) + \overline{b} \cdot (\eta_1 + \eta_2)}}{|\eta_1 + \eta_2|^{d-2}} \delta\binom{\tau+|\eta_1|^2+|\eta_2|^2}{\xi-\eta_1-\eta_2}\,\mathrm{d}\eta_1\mathrm{d}\eta_2
\]
and
\[
v(t,x)=e^{it\Delta}v_0(x)=C\int_{\R^d}e^{ix\cdot\zeta-it|\zeta|^2} e^{a|\zeta|^2 + b \cdot \zeta}  \,\mathrm{d}\zeta,
\]
we obtain the expression
\begin{align*}
&\widetilde{M\overline{\widetilde{uv}}}(t,x)v(t,x) \\
& = C\int_{\R^{3d}}e^{-ix\cdot(\eta_1+\eta_2-\zeta_1)+it(|\eta_1|^2+|\eta_2|^2-|\zeta_1|^2)} \frac{e^{a|\zeta_1|^2 + b \cdot \zeta_1} e^{\overline{a}(|\eta_1|^2+|\eta_2|^2) + \overline{b} \cdot (\eta_1 + \eta_2)}}{|\eta_1+\eta_2|^{d-2}} \,\mathrm{d}\eta \mathrm{d}\zeta_1.
\end{align*}
Hence for $t\in\R$ fixed,
\begin{align*}
e^{it\Delta}\left(\widetilde{M\overline{\widetilde{uv}}}(t,\cdot)v(t,\cdot)\right)(y) & =\int_{\R^{4d}}e^{-iy\cdot\zeta_2}  \delta(\eta_1+\eta_2-\zeta_1-\zeta_2)e^{it(|\eta_1|^2+|\eta_2|^2-|\zeta_1|^2-|\zeta_2|^2)} \\
& \qquad \qquad  \times \frac{e^{a|\zeta_1|^2 + b \cdot \zeta_1} e^{\overline{a}(|\eta_1|^2+|\eta_2|^2) + \overline{b} \cdot (\eta_1 + \eta_2)}}{|\eta_1+\eta_2|^{d-2}}  \,\mathrm{d}\eta \mathrm{d}\zeta,
\end{align*}
and integrating this expression with respect to $t$, it follows that the left-hand side of \eqref{e:ELnec} is equal to
\begin{align*}
C \int_{\R^{4d}}e^{-iy\cdot\zeta_2}\delta\binom{|\eta_1|^2+|\eta_2|^2-|\zeta_1|^2-|\zeta_2|^2}{\eta_1+\eta_2-\zeta_1-\zeta_2} 
\frac{e^{a|\zeta_1|^2 + b \cdot \zeta_1} e^{\overline{a}(|\zeta_1|^2+|\zeta_2|^2) + \overline{b} \cdot (\zeta_1 + \zeta_2)}}{|\zeta_1 + \zeta_2|^{d-2}}    \,\mathrm{d}\eta \mathrm{d}\zeta. 
\end{align*}
The integration with respect to $\eta \in \mathbb{R}^{2d}$ is carried out using Lemma \ref{l:convitself} (with $\beta = 0$), from which we see that the left-hand side of \eqref{e:ELnec} simplifies to 
\begin{align*}
C \int_{\R^{2d}} \bigg( \frac{|\zeta_1 - \zeta_2|}{|\zeta_1 + \zeta_2|}  \bigg)^{d-2}  e^{a|\zeta_1|^2 + b \cdot \zeta_1} e^{\overline{a}(|\zeta_1|^2+|\zeta_2|^2) + \overline{b} \cdot (\zeta_1 + \zeta_2)}    e^{-iy\cdot\zeta_2}  \,\mathrm{d}\zeta,
\end{align*}
and it follows that if \eqref{e:ELnec} holds, then
\[
\int_{\R^{d}} \bigg( \frac{|\zeta_1 - \zeta_2|}{|\zeta_1 + \zeta_2|}  \bigg)^{d-2}  e^{2\mathrm{Re}(a)|\zeta_1|^2 + 2\mathrm{Re}(b) \cdot \zeta_1}   \,\mathrm{d}\zeta_1  = C 
\]
for each $\zeta_2 \in \mathbb{R}^d$. This is false for $d \geq 3$ (and clearly true when $d=2$) which completes the proof of Theorem \ref{t:nobars}.
\end{proof}

\begin{acknowledgement}
This work was supported by the European Research Council [grant
number 307617] (Bennett) and the Engineering and Physical Sciences Research Council [grant number EP/J021490/1] (Bez, Pattakos).
\end{acknowledgement}

\end{document}